\documentclass[reqno]{amsart}
\usepackage[T2A]{fontenc}
\usepackage[utf8]{inputenc}
\usepackage[russian,english]{babel}
\usepackage{amssymb}  
\usepackage{amsmath} 
\usepackage{amsthm}  
\usepackage{stmaryrd} 
\usepackage{eucal}     

\advance\hoffset by -1in
\advance\voffset by -1in

\newtheorem{lemma}{Lemma}
\newtheorem{proposition}{Proposition}
\newtheorem{corollary}{Corollary}
\newtheorem*{problem}{Problem}
\newtheorem*{conjecture}{Conjecture}

\newcommand{\M}{\mathcal{M}}
\newcommand{\ld}{\leftthreetimes}
\newcommand{\llb} {\llbracket}
\newcommand{\rrb} {\rrbracket}

\renewcommand{\r}{\rho}

\title{On the commutative center of Moufang loops}

\author{Alexander N. Grishkov}
\address{Instituto de Matem\'atica e Estat\'\i stica, S\~ao Paulo, Brazil}
\email{shuragri@gmail.com}

\author{Andrei V. Zavarnitsine}
\address{Sobolev Institute of Mathematics, Novosibirsk, Russia}
\email{zav@math.nsc.ru}
\thanks{Supported by FAPESP (proc. 2017/14489-2)}

\date{}

\begin{document}
\begin{abstract} We construct two infinite series of Moufang loops of exponent $3$ whose commutative center (i.\,e. the set of elements
that commute with all elements of the loop) is not a normal subloop. In particular, we obtain examples of such loops of
orders $3^8$ and $3^{11}$ one of which can be defined as the Moufang triplication of the free Burnside group $\operatorname{B}(3,3)$.

{\sc Keywords:} Moufang loop, commutative center, normal subloop, Burnside group


{\sc MSC2010:}  20N05   

\end{abstract}
\maketitle

\section{Introduction}

A Moufang loop is a loop in which the identity
\begin{equation}\label{mi}
(xy)(zx)=(x(yz))x
\end{equation}
holds. Moufang loops are known to be {\em diassociative}, i.\,e. their subloops generated by a pair of
elements are groups. In particular, elements of Moufang loops have unique inverses and
$(xy)^{-1}=y^{-1}x^{-1}$ for all $x,y$.

The {\em commutative center} (also known as the {\em  Moufang center} or
the {\em commutant\/}\footnote{We avoid using the term
`commutant' due to the unfortunate collision with the term `derived subgroup' as both are
translated into Russian as `коммутант'.})
of a Moufang loop $M$ is the set
$$
\operatorname{C}(M)=\{c\in M\mid cx=xc\ \text{for all}\ x\in M\}.
$$
It is known \cite[Theorem IV.3.10]{pf} that $\operatorname{C}(M)$ is a subloop of $M$ which is characteristic.
The following problem had been open for quite a while and officially raised by A.\,Rajah
at {\em Loops '03} conference in Prague.

\begin{problem} Is $\operatorname{C}(M)$ a normal subloop of $M$?
\end{problem}

It was stated in \cite{gag} that the answer to this question is affirmative. Here we show that the answer is in fact
generally negative by constructing two series of examples of Moufang loops whose commutative center is not normal.
The first series is given by an explicit multiplication formula over any field of characteristic $3$. This series contains
a finite loop of order $3^{11}$. The second series is a particular case of the triplication of
Moufang loops of exponent $3$. We show that, for any Moufang loop $M$ of exponent $3$ that does not satisfy
the identity $[[x,y],z]=1$, there is a Moufang loop $\operatorname{M}(M,3)$ whose commutative center is not normal.
In particular, we have the example $\operatorname{M}(\operatorname{B}(3,3),3)$ of order $3^8$,
where $\operatorname{B}(3,3)$ is the free $3$-generator Burnside group of exponent $3$.

Recall that, for a Moufang loop $M$,
$$
\operatorname{Nuc}(M)=\{a\in M\mid (ax)y=a(xy)\ \text{for all}\ x,y\in M\}
$$
is the {\em nucleus} of $M$.
Our results reopen the following conjecture by S.\,Doro \cite{dor}:

\begin{conjecture} If $\operatorname{Nuc}(M)=1$ then $\operatorname{C}(M)$ is normal in $M$.
\end{conjecture}

\section{Motivation}

For elements $x,y,z$ of a Moufang loop $M$, we denote by $[x,y]$
the unique element $s\in M$ such that $xy=(yx)s$ and
by $(x,y,z)$ the unique element $t\in M$
such that $(xy)z=(x(yz))t$.

The idea behind the construction is the following straightforward observation.

\begin{lemma}\label{mot}
If $\operatorname{C}(M)$ is normal in $M$ then, for every $a\in \operatorname{C}(M)$
and every $b,c,d\in M$, we have $[(a,b,c),d]=1$.
\end{lemma}
\begin{proof}
  Since $\operatorname{C}(M)$ is normal and $a\in \operatorname{C}(M)$, we have $(a,b,c)\in \operatorname{C}(M)$.
\end{proof}

Therefore, for any Moufang loop $L$, if there is  $a\in \operatorname{C}(L)$
such that $[(a,b,c),d]\ne 1$ for some $b,c,d\in L$ then $\operatorname{C}(L)$ is not a normal
subloop in view of Lemma~\ref{mot}.
Examples of such loops are constructed in the next sections.

\section{Algebraic loop}\label{expol}

The underlying set of the loop is an $11$-dimensional vector space $V$ over a field $F$ of characteristic $3$.
An element $x\in V$ will be written as a tuple $x=(x_1,x_2,\ldots,x_{11})$, $x_i\in F$.
We introduce a new operation '$\circ$' on $V$ which is given,
for $x,y\in V$, by
\begin{equation}\label{mult}
x\circ y = x + y + f,
\end{equation}
where $f=(f_1,\ldots,f_{11})$ and $f_k$ are polynomials in $x_i$, $y_j$ explicitly given below.
\begin{align*}
f_1   &=  f_2 = f_3  = f_4 =  0, \\
f_5   &=  -x_3y_2,\
f_6    =    -x_4y_2,\
f_7    =    -x_4y_3,\\
f_8   &=\ \ \, x_1x_3y_2-x_1y_2y_3-x_2x_3y_1+x_2y_1y_3,\\
f_9   &=\ \ \, x_1x_4y_2-x_1y_2y_4-x_2x_4y_1+x_2y_1y_4,\\
f_{10}&=\ \ \, x_1x_4y_3-x_1y_3y_4-x_3x_4y_1+x_3y_1y_4,\\
f_{11}&=    -x_1x_2x_4y_3 +x_1x_2y_3y_4 +x_1x_3y_2y_4 +x_1x_4y_2y_3 +x_2x_3y_1y_4 +x_2x_4y_1y_3\\
      &\ \ \ -x_2y_1y_3y_4 +x_3x_4y_1y_2 -x_4y_1y_2y_3 -x_1x_5y_4 +x_1x_6y_3 -x_1x_7y_2 +x_1y_2y_7 \\
      &\ \ \ -x_1y_3y_6 +x_1y_4y_5 +x_2x_7y_1 -x_2y_1y_7 -x_3x_6y_1 +x_3y_1y_6 +x_4x_5y_1 -x_4y_1y_5\\
      &\ \ \  +x_8y_4 -x_9y_3 +x_{10}y_2.
\end{align*}

It can be checked that $L=(V,\circ)$ is a loop in which the Moufang identity (\ref{mi}) holds.
We omit the heavily technical verification. The identity element of $L$ is
the zero vector of $V$ and, for $x\in L$, we have
\begin{equation}\label{inv}
x^{-1} = -x + h,
\end{equation}
where $h=(h_1,\ldots,h_{11})$ and the polynomials $h_k$ are as follows:

\begin{align*}
  h_1    &= h_2 = h_3 =  h_4 = h_8 = h_9 = h_{10} = 0,\\
  h_5    &= -x_2x_3,\
  h_6     = -x_2x_4,\
  h_7     = -x_3x_4,\\
  h_{11} &= \phantom{-}x_2x_{10}-x_3x_9+x_4x_8.
\end{align*}

Let $e_1,\ldots,e_{11}$ be the standard basis of $V$, i.\,e.,
$e_i=(\ldots,0,1,0,\ldots)$ with `$1$' at the $i$th place. Define
$$
a=e_1,\quad b=e_2,\quad c=e_3,\quad d=e_4.
$$
\begin{lemma} $a\in \operatorname{C}(L)$.
\end{lemma}
\begin{proof}
This is true because using the multiplication formula (\ref{mult}) one can check that,
for an arbitrary $x\in L$, both $a\circ x$ and $x\circ a$ equal
$$
x+a+( 0, 0, 0, 0, 0, 0, 0, -x_2x_3, -x_2x_4, -x_3x_4,  x_2x_7-x_3x_6+x_4x_5).
$$
\end{proof}

\begin{lemma}\label{abcd} $[(a,b,c),d]\ne 1$ in $L$.
\end{lemma}
\begin{proof}
Using (\ref{mult}) it can be shown that the following equalities hold in $L$:
\begin{align*}
& e_5=[b,c],\ e_6=[b,d],\ e_7=[c,d], \\
& e_8=(a,b,c),\ e_9=(a,b,d),\ e_{10}=(a,c,d),\\
& e_{11}=[(a,b,c),d].\\
\end{align*}
Since $e_{11}\ne 0$ in $V$, it follows that $[(a,b,c),d]$ is not the identity element of $L$.
\end{proof}

Therefore, Lemma \ref{mot} implies
\begin{corollary} $\operatorname{C}(L)$ is not a normal subloop of $L$.
\end{corollary}

The following properties of $L$ are worth mentioning.
$L$ has exponent $3$ and satisfies the identity $[[x,y],z]=1$.
$\operatorname{C}(L)=\langle e_1,e_5,e_6,e_7,e_{11}\rangle_F$
is associative of order $3^5$, where $\langle \ \rangle_F$ denotes the $F$-linear span in $V$.
 $\operatorname{C}(L)$ is close to being normal --- it has a subloop $N$ of index $3$ which is normal in $L$.
We have $N=\langle e_5,e_6,e_7,e_{11}\rangle_F$ and $L/N$ is commutative.
Also, $\operatorname{Nuc}(L)=\langle e_8,e_9,e_{10},e_{11}\rangle_F$ and
$\operatorname{Z}(L)=\operatorname{Nuc}(L)\cap \operatorname{C}(L) = \langle e_{11}\rangle_F$.

In particular, if $F=\mathbb{F}_3$, we obtain an example of order $3^{11}$.

\section{Triplication of Moufang loops of exponent $3$}

A group $G$ possessing automorphisms $\rho$ and $\sigma$ that satisfy $\rho^3=\sigma^2=(\rho\sigma)^2=1$ is
called a {\it group with triality $S=\langle\rho,\sigma\rangle$} if
$$
(x^{-1}x^{\sigma})(x^{-1}x^{\sigma})^\rho(x^{-1}x^{\sigma})^{\rho^2}=1
$$
for every $x$ in $G$. In such a group,
the set
$$\M(G)=\{ x^{-1}x^{\sigma}\ |\ x\in G\}$$
is a Moufang loop with respect to the multiplication
\begin{equation} \label{loop_mult}
m.n=m^{-\rho} n  m^{-\rho^2}
\end{equation}
for all $n,m\in \M(G)$. Conversely, every Moufang loops arises so from a suitable group with triality.
Observe that taking powers or inverses of elements of $\M(G)$ is the same, whether considered
under the loop or group operation. We will also require the following properties.

\begin{lemma}[\mbox{\cite[Lemma 2.1]{mlt}}]\label{gzt}  Let $G$ be a group with triality
$S=\langle\rho,\sigma\rangle$.
Then, for all $m,n\in \M(G)$, we have
\begin{enumerate}
\item[$(i)$] $m^\sigma = m^{-1}$;
\item[$(ii)$] $m,m^\r,m^{\r^2}$ pairwise commute;
\item[$(iii)$] $m^{-\r} nm^{-\r^2}=n^{-\r^2} m n^{-\r}$;
\item[$(iv)$] $m.n.m = mnm$.
\end{enumerate}
\end{lemma}

For more details on groups with triality, see \cite{dor,gz}.

In the loop $(\M(G),.)$, we write $\llb x,y\rrb=x^{-1}.y^{-1}.x.y$ instead of $[x,y]=x^{-1}y^{-1}xy$ to make
a distinction between the loop and group commutator.

Let $G$ be a group with triality $S=\langle\rho,\sigma\rangle$ and let $M=\M(G)$. It was observed in \cite[Lemma 3]{syl}
that the natural semidirect product $\widetilde{G}=G\ld \langle\rho \rangle$ is also a group
with triality $S$ if and only if $M$ has exponent $3$. In this case, we will denote the
corresponding Moufang loop $\M(\widetilde{G})$ by $\operatorname{M}(M,3)$ and
call it the {\em triplication} of $M$. It contains $M$ as a normal subloop of index $3$
and coincides as a subset of $\widetilde{G}$ with
$$
M\ \cup\  \rho M\rho\ \cup\ \rho^2 M\rho^2.
$$

\begin{lemma}\label{trpr}
Let $M$ be a Moufang loop of exponent $3$ and let $L=\operatorname{M}(M,3)$. Then
\begin{enumerate}
\item[$(i)$] $L$ has exponent $3$;
\item[$(ii)$] The multiplication rule of $L$ is as given in Table \ref{mmm3};
\item[$(iii)$] $\r \in \operatorname{C}(L)$;
\item[$(iv)$] $(\r,m,n)=\llb n^{-1},m^{-1}\rrb$ for all $m,n\in M$.
\item[$(v)$] $L$ is associative iff $M$ is an abelian group.

\end{enumerate}
\end{lemma}
\begin{proof} We will use Lemma \ref{gzt} and the fact that $M$ has exponent $3$ implicitly throughout the proof.

$(i)$
Let $m\in L$. If $m\in M$ then $m^3=1$. If $m=\r n\r$ then
$m^3=\r n \r^2 n \r^2 n \r = n^{\r^2}n n^{\r}=1$. If $n=\r^2 n\r^2$ then $m^3=\r^2 n\r n\r n\r^2=n^\r nn^{\r^2}=1$. Hence, $L$ has
exponent $3$.

$(ii)$ Let $m,n\in M$. Then the product rules are
\begin{multline*}
    m.\r n\r = m^{-\r}\r n\r m^{-\r^2} = \r m^{-\r^2}nm^{-\r}\r=\r (n.m)\r, \\
\shoveleft{\r m\r.n = n^{-\r^2}\r m\r n^{-\r} = \r n^{-1}mn^{-1} \r = \r (n^{-1}.m.n^{-1}) \r = \r (m^{-1}.n.m^{-1}) \r,}\\
\shoveleft{\r m\r . \r n\r = (\r m\r)^{-\r}\r n\r (\r m\r)^{-\r^2}\!=
\r m^{-1}\r n\r m^{-1}\r = \r^2 m^{-\r}nm^{-\r^2}\r^2 }\\
\shoveright{= \r^2(m.n)\r^2,}\\
\shoveleft{\r^2 m\r^2.\r n \r = (\r^2 m\r^2)^{-\r}\r n\r (\r^2 m\r^2)^{-\r^2}
=\r m^{-\r}\r^2n\r^2 m^{-\r^2}\r = m^{-1}nm^{-1}}\\
    =m^{-1}.n.m^{-1},
\end{multline*}
and similarly for other choices of the factors.

$(iii)$ Let $m\in M$. Using $(ii)$, we have
\begin{align*}
       \r.n & =\r^2\r^2.n=\r^2n\r^2, &       n.\r & =n.\r^2\r^2= \r^2 n^{-2}\r^2 = \r^2 n\r^2, \\
  \r.\r n\r & =\r^2\r^2.\r n\r=n,     &  \r n\r.\r & =\r n\r.\r^2\r^2 = n,  \\
  \r. \r^2 n \r^2& =\r^2\r^2. \r^2 n \r^2 =\r n\r, & \r^2 n \r^2.\r & = \r^2 n \r^2.\r^2\r^2 = \r n\r.
\end{align*}
Therefore, $\r \in \operatorname{C}(L)$.

$(iv)$ By $(ii)$, we have
\begin{multline*}
(\r,m,n)=(\r.(m.n))^{-1}.((\r.m).n)=(\r^2(m.n)\r^2)^{-1}.(\r^2m\r^2.n)\\
=(\r(m.n)^{-1}\r).(\r^2(n.m)\r^2)=n.m.(m.n)^{-1}= \llb n^{-1},m^{-1}\rrb.
\end{multline*}

$(v)$  If $M$ is not associative then neither is $L$, since $M$ is a subloop of $L$. By $(iv)$, $L$ is not
associative if $M$ is not commutative.
Conversely, if $M$ is a commutative group then Table 1 turns into the multiplication rule of
the direct product $M\times \langle\r \rangle$, hence $L$ is associative in this case (an elementary abelian $3$-group, in fact).
\end{proof}

\begin{table}
  \caption{The product $x.y$ in $\operatorname{M}(M,3)$}\label{mmm3}
\begin{center}
\begin{tabular}{c|ccc}
   $x\ \ \backslash \ \ y$ & $n$ & $\r n\r$ & $\r^2 n \r^2$ \\
\hline
$m$ & $m.n$ & $\r (n.m) \r$ & $\r^2 (m^{-1}.n.m^{-1}) \r^2\vphantom{A^{A^A}}$ \\
$\r m \r$ & $\r (m^{-1}.n.m^{-1}) \r$ & $\r^2 (m.n) \r^2$ & $n.m$  \\
$\r^2 m \r^2$ & $\r^2(n.m) \r^2$ & $m^{-1}.n.m^{-1}$ & $\r (m.n)\r$
\end{tabular}
\end{center}
\end{table}

Lemma \ref{trpr}$(ii)$ shows that up to isomorphism $\operatorname{M}(M,3)$
does not depend on the group $G$ as long as $\M(G)\cong M$ (such a group is not unique). In other words, we can
construct $\operatorname{M}(M,3)$ by adjoining to $M$ an outer element $\r=\r^2\r^2$ and use Table \ref{mmm3} to define
multiplication on formal elements of the shape $m,\r m\r, \r^2 m\r^2$ with $m\in M$.

\begin{proposition}\label{main}
  Let $M$ be a Moufang loop of exponent $3$ that does not satisfy the identity $[[x,y],z]=1$. Then the
  commutative center of $\operatorname{M}(M,3)$ is not a normal subloop.
\end{proposition}
\begin{proof}
Denote $L=\operatorname{M}(M,3)$. By Lemma \ref{trpr}$(iii)$, $\r\in \operatorname{C}(L)$. By assumption,
there are $x,y\in M$ such that $[y^{-1},x^{-1}]\not\in \operatorname{C}(L)$.
We have $(\r,x,y)=[y^{-1},x^{-1}]$ by Lemma \ref{trpr}$(iv)$. Therefore, $\operatorname{C}(L)$ is not normal
by Lemma \ref{mot}.
\end{proof}

It is known \cite[\S 18.2]{hall} that the free $r$-generator Burnside group $\operatorname{B}(3,r)$ of
exponent $3$ has nilpotency class $3$ for $r\geqslant 3$.
In particular, it does not satisfy the identity $[[x,y],z]=1$. Hence, we obtain

\begin{corollary} The commutative center of the Moufang triplication $\operatorname{M}(\operatorname{B}(3,r),3)$
is not a normal subloop if $r\geqslant 3$.
\end{corollary}

The minimal example in this series is the loop $\operatorname{M}(\operatorname{B}(3,3),3)$ of order $3^8$.

Finally, observe that the loops from Section \ref{expol} cannot be obtained by an application of Proposition \ref{main}, because
they satisfy the identity $[[x,y],z]=1$.

\end{document}